\documentclass[reqno,11pt]{amsart}
\usepackage{hyperref,color}
\usepackage{amssymb}
\usepackage[centering]{geometry}

\newtheorem{theorem}{Theorem}
\newtheorem{corollary}{Corollary}
\newtheorem{proposition}{Proposition}
\newtheorem{lemma}{Lemma}
\newtheorem{hypothesis}{Hypothesis}
\theoremstyle{remark}
\newtheorem{remark}{Remark}
\newtheorem{example}{Example}

\newcommand\pp{\boldsymbol{\phi}}
\newcommand\vp{\boldsymbol{\varphi}}
\newcommand\uu{\boldsymbol{u}}
\newcommand\zz{\boldsymbol{z}}

\newcommand\Xe{X_\varepsilon}
\newcommand\cC{\mathcal{C}}
\newcommand\R{\mathbb{R}}
\newcommand\HH{\boldsymbol{H}}
\newcommand\EE{\boldsymbol{E}}
\newcommand\cE{\mathcal{E}}
\newcommand\cP{\mathcal{P}}
\newcommand\cM{\mathcal{M}}

\newcommand\cA{\mathcal{A}}
\newcommand\cH{\mathcal{H}}
\newcommand\cI{\mathcal{I}}

\newcommand\cJ{\mathcal{J}}

\newcommand\cK{\mathcal{K}}

\newcommand\cZ{\mathcal{Z}}
\newcommand\cY{\mathcal{Y}}
\newcommand\px{\partial_x}

\newcommand\ud{\, \textnormal{d}}
\newcommand\udd{\textnormal{d}}
\newcommand{\inprod}[2]{\left\langle{#1},{#2}\right\rangle}

\DeclareMathOperator\sgn{sgn}
\DeclareMathOperator\sech{sech}

\title[Odd kink dynamics with one internal mode]{Kink dynamics under odd perturbations 
for (1+1)-scalar field models with one internal mode}
\author{Micha{\l} Kowalczyk}
\address{Departamento de Ingenier\'{\i}a Matem\'atica and Centro
de Modelamiento Matem\'atico (UMI 2807 CNRS), Universidad de Chile, Casilla
170 Correo 3, Santiago, Chile.}
\email {kowalczy@dim.uchile.cl}

\author{Yvan Martel}
\address{CMLS, \'Ecole polytechnique, CNRS, Institut Polytechnique de Paris, 91128 Palaiseau Cedex, France}
\email{yvan.martel@polytechnique.edu}

\thanks{M. Kowalczyk was partially funded by Chilean research grants FONDECYT 1210405 and ANID projects ACE210010 and FB210005. Part of this work was done when he was visiting the CMLS at \'Ecole Polytechnique, France in the autumn of 2020. The authors thank Claudio Mu\~noz (Universidad de Chile) for useful discussions.}


\begin{document}
\begin{abstract}
We consider odd symmetric (1+1)-scalar field models with one internal mode.
Under natural and robust assumptions, including the Fermi golden rule,
we prove the asymptotic stability of the kink by odd perturbations in the energy space. 
For example, the result applies to the $\phi^4$ model with odd symmetry and some of its perturbations.
\end{abstract}

\maketitle

\section{Introduction}

\subsection{Main result}
Consider the one-dimensional scalar field model
\begin{equation}\label{eq:phi}
\partial_{tt} \phi - \partial_{xx} \phi + W'(\phi) =0 ,\quad (t,x)\in \R\times \R.
\end{equation}
As a first order system for $\pp=(\phi,\phi_t)=(\phi_1,\phi_2)$, this problem reads
\begin{equation}
\label{eq:pp}
\begin{cases}
\dot \phi_1= \phi_2\\
\dot \phi_2= \partial_{xx} \phi_1 -W'(\phi_1).
\end{cases}
\end{equation}
The potential $W$ satisfies the following assumptions
\begin{equation}\label{on:W}
\begin{cases} 
\mbox{$W:\R\to [0,\infty)$ is of class $\cC^\infty$,}\\
\mbox{$W$ is even, $W>0$ on $(-1,1)$,}\\
\mbox{$W(\pm 1)=0,$ $W'(\pm 1)=0$, $W''(\pm 1)>0$.}
\end{cases}
\end{equation}
By change of unknown $\phi\to a \phi$, fixing the location of the zeros of $W$ at $\pm 1$ does not restrict the generality.
By change of variable $(x,t)\to (b x,b t)$, it would also be possible to fix $W''(\pm 1)$ to the value $1$, but for the sake of clarity, we simply denote
\[ W''(\pm 1)=\omega^2 \quad \mbox{for $\omega>0$.} \]
Under the above assumptions, equation~\eqref{eq:phi} admits a static solution $\HH=(H,0)$, called \emph{kink}, where the function $H$ is the unique, odd and smooth solution of the equation
\begin{equation*}
\begin{cases}
H'' = W'(H) \mbox{ on $\R$},\\
\lim_{\pm\infty} H = \pm 1.
\end{cases}
\end{equation*}
Note that it satisfies $H'>0$ on $\R$. Moreover, for any $k\geq 0$, there exists $C_k>0$ such that 
\begin{equation*}
|H(x)\mp 1|\leq C_0 e^{\mp \omega x},\quad |H^{(k)}(x)|\leq C_k e^{-\omega |x|} \quad \mbox{on $\R$}
\end{equation*}
(see \emph{e.g.} \cite[Lemma 1]{KMMV}).
Recall the conservation laws (energy and momentum) of the model
\begin{align*}
\cE[\pp] & = \frac 12 \int \left[ \phi_2^2 + (\partial_x \phi_1)^2 + 2 W(\phi_1)\right],\\
\cP[\pp] & = \int \phi_2 \partial_x \phi_1.
\end{align*}
The energy space is
\[
\EE=\left \{\pp\in L^1_{\rm loc}(\R)\times L^1_{\rm loc}(\R) :
\partial_x \phi_1 \in L^2(\R),\ \sqrt{W(\phi_1)}\in L^2(\R),\ \phi_2\in L^2(\R)\right\}.
\]
The space of odd functions in $\EE$ is invariant under the flow since $W'$ is odd.
To study the stability properties of $\HH$ under odd perturbations, we introduce the following subset $\EE^{odd}_{\HH}$ of $\EE$\[
\EE^{odd}_{\HH} = \left\{\pp\in L^1_{\rm loc}(\R)\times L^1_{\rm loc}(\R) :
\mbox{$\pp$ is odd and }
\partial_x \phi_1,\ \phi_1-H,\ \phi_2\in L^2(\R)\right\}.
\]
Recall from~\cite{MR678151} (see also \cite{KMMV,Lohe}) the stability of the kink by odd perturbations of the initial data in the energy space:
\emph{for any $\pp^{in}\in \EE^{odd}_{\HH}$ with
$\|\pp^{in}-\HH\|_{H^1\times L^2}$ small, 
the solution $\boldsymbol\phi$ of \eqref{eq:phi}
with $\pp(0)=\pp^{in}$ is global and, for a constant $C>0$,}
\begin{equation}\label{eq:stab}
\sup_{t\in\R}\|\pp(t)-\HH\|_{H^1\times L^2}
\leq C \|\pp^{in}-\HH\|_{H^1\times L^2}.
\end{equation}
In the framework of the above stability result, we decompose
\begin{equation}\label{eq:ppp}
\boldsymbol \phi=\HH+\vp
\end{equation}
and we write \eqref{eq:pp} in terms of $\vp=(\varphi_1,\varphi_2)$
\begin{equation}
\label{eq:varphi}
\begin{cases}
\dot \varphi_1= \varphi_2\\
\dot \varphi_2=- L_0 \varphi_1 - N 
\end{cases}
\end{equation}
where $ L_0$ is the linearized operator around the kink
\begin{equation*}
 L_0 = -\partial_{xx} + W''(H)
\end{equation*}
and $N$ gathers nonlinear terms
\begin{equation}\label{def:N}
N =W'(H+\varphi_1) - W'(H)- W''(H) \varphi_1.
\end{equation}

\smallskip

The present paper is devoted to the question of \emph{asymptotic stability}, which requires sharp information about the spectral properties of $L_0$ and possibly nonlinear conditions.
Recall that the operator $ L_0$ has the continuous spectrum $[\omega^2,+\infty)$ and a non empty discrete spectrum since by differentiating the equation of $H$, one finds $L_0 H'=0$.
In \cite[Theorem~2]{KMMV}, a sufficient condition on the potential $W$ was given to ensure asymptotic stability of the kink, without the oddness restriction. Indeed, introducing the transformed potential
\[
V = - W (\log W)''
\]
it is sufficient that $V'\not\equiv 0$ and $V$ is \emph{repulsive} to obtain the asymptotic stability of the kink in the energy space in some local sense. This condition implies that $ L_0$ has no eigenvalue other than $0$ and no resonance at $\omega^2$.
While several concrete applications of this criterion were given in \cite[Sect. 5]{KMMV},
some classical models from the Physics literature, like the $\phi^4$ model discussed later in Sect.~\ref{S:1.4},
do not enter this framework.
Our objective here is to deal with general models of the form \eqref{eq:phi}, admitting at least one \emph{internal mode}, \emph{i.e.} for which the operator $L_0$ has a non trivial spectrum.
This will extend, with a different proof, the result obtained in~\cite{KMM} for the $\phi^4$ model.

\begin{hypothesis}[Presence of an internal mode]
In addition to the eigenvalue~$0$ associated to the even eigenfunction $H'$, suppose that the operator $L_0$ has at least one eigenvalue in $(0,\omega^2)$, denoted by $\lambda^2$ with $\lambda\in (0,\omega)$, associated to an odd normalized eigenfunction~$Y$.
\end{hypothesis}
The existence of an additional eigenvalue for $ L_0$ is a serious difficulty for proving the asymptotic stability of the kink. Indeed, the function
\[\vp(t,x)=(\sin(\lambda t) Y(x), \lambda\cos(\lambda t) Y(x))\]
is then a solution of the linear counterpart of~\eqref{eq:varphi}
\begin{equation*}
\begin{cases}
\dot \varphi_1= \varphi_2\\
\dot \varphi_2=- L_0 \varphi_1.
\end{cases}
\end{equation*}
In the presence of such \emph{time periodic} linear solutions, the asymptotic stability property should rely on a \emph{nonlinear condition}, called the (nonlinear) \emph{Fermi golden rule}, following the terminology of \cite{SoWe2}.
We formulate this hypothesis for the internal mode $Y$.

\begin{hypothesis}[Fermi golden rule]
Suppose that there exists an odd, bounded function $g:\R\to\R$ of class $\cC^\infty$, solution of
\[L_0 g = 4\lambda^2 g\]
and satisfying
\begin{equation}
\label{fgr}
\int W'''(H) Y^2 g \neq 0.
\end{equation}
\end{hypothesis}

The presence of more than one internal mode would further complicate the analysis
(see references in Sect.~\ref{S:1.3}) and we do not investigate this issue here.
Thus, we will introduce one last hypothesis that constraints the number of odd internal modes. To formulate this spectral assumption, we introduce the iterated Darboux factorization associated to $L_0$. We use the convenient formulation in \cite[Proposition 1.9]{CuMa_1}, \cite[Sect. 3]{deift-trub} (see also \cite{CGNT,KMMV}).
Introducing the differential operators 
\begin{equation*}
U_0= H' \cdot \partial_x\cdot (H')^{-1} , \quad U_0^\star= -(H')^{-1} \cdot \partial_x\cdot H',
\end{equation*}
we observe that $L_0=U_0^\star U_0$ and we set
\begin{equation*}
L_1=U_0 U_0^\star = -\partial_{xx} +  2 \left(\frac{H''}{H'}\right)^2 - \frac {H'''}{H'} 
=-\partial_{xx} + P_1.
\end{equation*}
By \cite{CuMa_1,deift-trub}, 
since $L_0$ has the second eigenvalue $\lambda^2\in (0,\omega^2)$, the operator $L_1$ has the principal eigenvalue $\lambda^2$, associated to the even eigenfunction $Z=U_0Y$, $Z>0$ on $\R$.
Set
\[
U_1= Z \cdot \partial_x\cdot Z^{-1}, \quad U_1^\star=-Z^{-1}\cdot\partial_x\cdot Z.
\]
We observe that $L_1 = \lambda^2 + U_1^\star U_1$
and we set 
\[
L_{2}=\lambda^2 + U_1U_1^\star 
=-\partial_{xx} +  \lambda^2 + 2\frac{(Z')^2}{Z^2} - \frac{Z''}{Z} = -\partial_{xx}+ P_2.
\]
Note that $P_1$ and $P_2$ are smooth, even, bounded functions with $\lim_{\pm \infty}P_1=\lim_{\pm \infty}P_2=\omega^2$.
With such notation, we observe the following crucial \emph{conjugation identity}
\begin{equation}\label{eq:id}
U_1 U_0 L_0 = L_2 U_1 U_0.
\end{equation}
Concerning the potential $P_2$, we will need the following assumption.

\begin{hypothesis}[Spectral assumption]
Suppose that there exists a constant $\gamma>0$ such that the operator 
$-(1-\gamma) \partial_{xx} +\frac 12xP_2'$ has at most one negative eigenvalue.
\end{hypothesis}

This hypothesis means that the operator $-(1-\gamma) \partial_{xx} +\frac 12xP_2'$
has no negative eigenvalue associated to an odd eigenfunction.
This is the analogue of a repulsivity condition for the potential $P_2$, restricted to odd functions.
In particular, for the operator $L_0$, it leaves the possibility of existence of a third eigenvalue, associated to an even eigenfunction, but not of an odd eigenfunction other than $Y$.

\smallskip

The main result of this article is the asymptotic stability of the kink of~\eqref{eq:phi} 
with respect to odd perturbations in the energy space under Hypothesis~1,~2 and~3.

\begin{theorem}\label{TH:1}
Assume \eqref{on:W} and Hypothesis~1,~2 and~3.
There exists $\delta>0$ such that for any 
$\pp^{in}\in \EE^{odd}_{\HH}$ with
$\|\pp^{in}-\HH\|_{H^1\times L^2} \leq \delta$,
the global solution $\boldsymbol\phi$ of \eqref{eq:phi}
with $\pp(0)=\pp^{in}$ satisfies, for any bounded interval $I$ of $\R$,
\begin{equation*}
\lim_{t\to \pm \infty} \|\pp(t)-\HH\|_{H^1(I)\times L^2(I)}=0.
\end{equation*}
\end{theorem}

We also check that our set of hypothesis is robust by perturbation of the model.
Let $\eta_0>0$ be a small parameter to be fixed later.
We consider 
\begin{equation}\label{def:Weta}
W_\eta = (1+\eta) W
\end{equation}
where the function $\eta$ satisfies
\begin{equation}\label{on:eta}
\begin{cases} 
\mbox{$\eta:\R\to \R$ is of class $\cC^\infty$,}\\
\mbox{$\eta$ is even,}\\
\mbox{$\sup_{[-1,1]}|\eta^{(k)}| \leq \eta_0$ for $k\in \{0,\ldots,4\}$.}
\end{cases}
\end{equation}
We refer to \cite[Remark 4.3]{KMMV} for a justification of the choice of a multiplicative perturbation of the potential.

\begin{corollary}\label{TH:2}
Assume \eqref{on:W} and Hypothesis~1,~2 and~3.
There exists $\eta_0>0$ such that for any function $\eta$ satisfying 
\eqref{on:eta}, the potential $W_\eta$ defined in~\eqref{def:Weta} also satisfies \eqref{on:W} and Hypothesis~1,~2 and~3. In particular, the corresponding kink is asymptotically stable by odd perturbations in the energy space.
\end{corollary}

\begin{remark}\label{RK:1}
As noticed in \cite[Remark (1), p.14]{SoWe2}, Hypothesis~2 and~3 imply that
\[
\frac {\omega}2 < \lambda < \omega.
\]
Indeed, assuming the existence of a non zero, bounded and odd function $g$ such that
$L_0 g = 4\lambda^2 g$ with $4\lambda^2\leq \omega^2$, the function $g_2 = U_1 U_0 g$ satisfies
$g_2\not \equiv 0$ and 
$L_2 g_2 = 4 \lambda^2 g_2$, which contradicts Hypothesis~3 by a virial argument
(see \emph{e.g.} \cite[Appendix]{GP} for a similar argument).
The fact that $g_2\not \equiv0$ follows from the equivalence
($c_1$ and $c_2$ are constants)
\begin{equation}\label{eq:equiv}
U_1U_0 g =0\iff U_0g=c_1 Z\iff g = c_2Y.
\end{equation}
\end{remark}

\subsection{References}\label{S:1.3}

The question of asymptotic stability of kinks for scalar field models is closely related to the one of global existence and asymptotic behavior of small solutions to nonlinear Klein-Gordon equations. These questions have been studied by many different techniques and authors, both in the case of constant and variable coefficients. We refer to the pioneering works \cite{Delort,Delort_fourier,K1,K2,Shatah1}. For early asymptotic stability results for nonlinear Schr\"odinger models, see~\cite{bus_per1,bus_per2}.

The asymptotic stability of kinks for one-dimensional scalar field models without internal mode has been studied in different settings and by different methods in~\cite{KK2,KMMV}.
A related question is the (conditional) asymptotic stability of solitons for the focusing nonlinear Klein-Gordon equation~\cite{BCS,KMM4,KNS}. 
Several results on the \emph{full asymptotic stability question} (among other estimates, convergence to $0$ of the solution is obtained in the supremum norm on $\R$, with an explicit decay rate) for one-dimensional nonlinear Klein-Gordon models were obtained recently in various relevant cases in~\cite{GP,LS1,LS2,LS3,Ste}.

For wave type equations, the case where the underlying linear operator admits an internal mode was first considered in the seminal paper \cite{SoWe2}, in the three dimensional setting and for cubic nonlinearities. This was recently extended to quadratic nonlinearities in \cite{LP}.
We also refer to \cite{MR2373326,DM,KK1,KMM} for results concerning the $\phi^4$ model or  higher dimensional related models.
For the integrable sine-Gordon equation, we refer for instance to~\cite{AMP2,denzler,Gong,LuS}.
See~\cite{JKL} for the construction and classification of kink-antikink pairs for the $\phi^4$ model.
Recall that general Klein-Gordon models with an arbitrarily large number of internal modes were considered in \cite{Bam_Cucc}. For similar questions in the case of nonlinear Schr\"odinger models, we refer to~\cite{CuMa_1}. See also the reviews \cite{CuMa_review,KMM3} for more references.
When a resonance exists, rather than an internal model, the situation is much more delicate. Ways to deal with such situation with some generality were recently found in~\cite{LLSS,LLS,LLS2}.
Last, we refer to~\cite{BDKK} for cases where the Fermi golden rule does not hold.

\subsection{Applications}\label{S:1.4}

\begin{example}[The $\phi^4$ model]
For the potential
\[
W(\phi)=\frac{1}{4}(\phi^2-1)^2,
\]
the model \eqref{eq:phi} is known as the $\phi^4$ equation; see
the general references~\cite{Kevrekidisbook,kruskal_segur,MaSutbook}.
We recall that the asymptotic stability of the kink $H(x)=\tanh(x / \sqrt{2})$ under odd perturbations  was first proved in \cite{KMM}.
We check that Hypothesis 1, 2 and 3 hold for this model, so that Theorem~\ref{TH:1} provides another proof for this model.
The operator $L_0$ is 
\[
L_0=-\partial_{xx}+2-3\sech\left(\frac{x}{\sqrt{2}}\right).
\]
It is well-known that the discrete spectrum of $L_0$ consists of simple eigenvalues $0$ and $\lambda^2=\frac{3}{2}$,
with normalized eigenfunctions given by
\begin{equation*}
H'(x)= \frac 1 {\sqrt{2}} \sech^2\Big(\frac{x}{\sqrt{2}}\Big),
\quad 
Y(x)= c\tanh\left(\frac{x}{\sqrt{2}}\right) \sech\left(\frac{x}{\sqrt{2}}\right).
\end{equation*}
Thus, Hypothesis 1 is satisfied with $\lambda^2 \in (\frac 14{\omega^2},\omega^2)$, where $\omega^2=2$.
Moreover, Hypothesis~2 is satisfied for the $\phi^4$ model. Indeed, 
it is known by~\cite{Segur} that the function $g$ defined by
\[
g(x)=\sin(2x)\left(1+\frac 12 \sech^2\left(\frac{x}{\sqrt{2}}\right)\right)+\sqrt{2}\cos(2x)\tanh\left(\frac{x}{\sqrt{2}}\right)
\]
is a solution of $L_0 g = 6g$. 
It was checked by an explicit calculation in~\cite[Appendix~14]{DM} that \eqref{fgr} is satisfied for $g$.
Last, in the case of the $\phi^4$ model, it is straightforward to check that $P_2\equiv 2$ after the Darboux factorization. Therefore, Hypothesis~3 holds in this case.
Compared to the one in \cite{KMM}, the present proof is more general and allows to get rid of numerical computations of some constants.
\end{example}

\begin{example}[The $\phi^8$ model in the $\phi^4$ limit]
We consider the $\phi^8$ model (see for example \cite{Gani2,Khare,Lohe}), \emph{i.e.} equation~\eqref{eq:phi} with the potential
\[
W(\phi)=\frac{1}{4}(\phi^2-1)^2(\phi^2-m^2)^2,
\]
where $m>1$ is a parameter. Changing the time and space variables
$(t,x)\mapsto (m^2t,m^2x)$, we obtain the scaled potential
\[
W_m(\phi)=\frac 1{m^4} \frac{1}{4}(\phi^2-1)^2(\phi^2-m^2)^2
=\frac{1}{4}(\phi^2-1)^2\left[ 1- \frac{\phi^2}{m^2} \left(2-\frac{\phi^2}{m^2}\right)\right]
\]
which can be seen as a perturbation of the $\phi^4$ potential for $m$ large.
Applying Corollary~\ref{TH:2}, we see that the odd kink of the $\phi^8$ model is asymptotically stable
for $m$ large enough.
\end{example}

\subsection{Outline of proof}
We work in the framework of the stability property:
for $\delta>0$ small enough to be chosen, 
let $\pp^{in}\in \EE^{odd}_{\HH}$ with
$\|\pp^{in}-\HH\|_{H^1\times L^2} \leq \delta$ and 
let $\boldsymbol\phi$ be the global solution of \eqref{eq:phi}
with $\pp(0)=\pp^{in}$.
The proof of Theorem~\ref{TH:1} starts by decomposing $\pp$ as in \eqref{eq:ppp}, so that 
$(\varphi_1,\varphi_2)$ is a global solution of \eqref{eq:varphi} in $H^1\times L^2$.
By the stability result \eqref{eq:stab}, we have $\sup_{\R} \{\|\varphi_1\|_{H^1} + \|\varphi_2\|\}_{L^2}\leq C  \delta$.
We further decompose $(\varphi_1,\varphi_2)$ using the eigenfunction $Y$ from Hypothesis 1, 
setting
\begin{equation*}
\varphi_1= u_1+z_1 Y, \quad
\varphi_2=u_2+ \lambda z_2 Y, 
\end{equation*}
where the time-dependent function $\zz=(z_1,z_2)$ is chosen so that, for all $t\in \R$,
\[
\langle u_1, Y\rangle =\langle u_2, Y\rangle=0.
\]
Moreover, denoting $|\zz|=\sqrt{z_1^2+z_2^2}$, we have
\begin{equation}\label{eq:unif}
\sup_{\R} \left\{\|u_1\|_{H^1} + \|u_2\|_{L^2} + |\zz| \right\}\leq C \delta.
\end{equation}
By \eqref{eq:varphi} and the normalization $\langle Y,Y\rangle=1$, the following system holds
\begin{equation}
\label{eq:zu}
\begin{cases}
\dot z_1=\lambda z_2\\
\dot z_2=-\lambda z_1 - \frac 1\lambda\langle N, Y\rangle \\
\dot u_1= u_2\\
\dot u_2=-L_0 u_1 - (N-\langle N, Y\rangle).
\end{cases}
\end{equation}
The proof of Theorem~\ref{TH:1} involves localized virial arguments for Klein-Gordon equations
inspired from \cite{KMM,KMM4,KMMV}
(see also \cite{MMjmpa} and \cite{MR} for the introduction of virial estimates in similar contexts for the generalized Korteweg-de Vries equations (gKdV) and the nonlinear Schr\"odinger equations (NLS)), combined with Darboux factorization (see \cite{deift-trub,KMM4,KMMV,KNS,RR} for wave-type equations and \cite{CuMa_1,Mnls} for NLS type equations).

\smallskip

The plan of the article is the following. In Sect.~2, a first virial estimate on the function $\uu$, solution of \eqref{eq:zu}, provides a preliminary estimate of $\uu$ at any large scale, with a decoupling of the internal mode component 
$\zz$ at the linear order. In Sect.~3, we control the internal mode component $\zz$
using Hypothesis 2. This part contains the main new ingredient of this paper, which is the introduction of the functional $\mathcal J$ in the proof of Proposition~\ref{PR:2}.
In Sect.~4, we recall the properties of a regularized version of the operator $U_1U_0$, used to define a \emph{transformed problem} related to the algebraic observation~\eqref{eq:id}. In Sect.~5, we prove estimates on the transformed problem using a second virial argument based on Hypothesis~3. Last, in Sect.~ 6, we close the estimates.
This strategy based on two virial arguments was first introduced in \cite{KMM4}. Here, it is combined with a treatment of the internal mode inspired by \cite{KMM} (see also Remark~\ref{RK:2}).
Corollary \ref{TH:2} is proved in Sect.~7.

\smallskip

We remark that the odd symmetry is supposed for simplicity. Following \cite{KMMV}, assuming Hypothesis 1, it is also possible to formulate sufficient spectral conditions for asymptotic stability of the kink for general initial data.
However, this would not include the case of the $\phi^4$ equation without symmetry, because of the presence of an even resonance for this model.

\subsection*{Notation}
In this paper, the notation $F \lesssim G$ means that $F\leq CG$ for a constant $C>0$ independent of $F$ and $G$. 
Unless otherwise indicated, the implicit constants $C>0$ are supposed to be independent of the parameters $A,B,\varepsilon$ and $\delta$ introduced below.

We will use the following notation for
real-valued functions $u$, $v\in L^2(\R)$,
\[
\inprod{u}{v}=\int_\R u(x) v(x) \ud x,\quad
\|u\|=\sqrt{\inprod{u}{u}}.
\]
We fix the constant $\kappa>0$ by
\begin{equation*}
\kappa=\frac{1}{12}\sqrt{\omega^2-\lambda^2}
\end{equation*}
and we set
\begin{equation*}
\rho(x)=\sech^2(\kappa x).
\end{equation*}
We also define the space $\cY$ of $C^\infty$ functions $f:\R\to\R$ that satisfy: for any $k\geq 0$, 
there exists $C_k>0$ such that on $\R$,
\begin{equation*}
|f^{(k)}|\leq C_k \rho^{5}.
\end{equation*}
Note for example that $H'$, $Y$ and $Z$ belong to $\cY$.

For any function $F$, we denote (recall the normalization $\|Y\|=1$)
\[
F^\perp=F- {\langle F,Y\rangle} Y,\quad \langle F^\perp,Y\rangle =0.
\]

\subsubsection*{Regularization of the Darboux factorization}

Define the operator $\Xe \colon L^2(\R)\to H^2(\R)$, $\Xe = (1-\varepsilon \partial_{xx})^{-1}$ via its Fourier transform representation, for $h\in L^2$,
\[
\widehat {\Xe h}(\xi)=\frac{\hat h(\xi)}{1+\varepsilon \xi^2}.
\]
We define the bounded operator on $L^2$
\[
S_\varepsilon=\Xe U_1 U_0.
\]

\subsubsection*{Notation for virial arguments}
We consider a smooth even function $\chi:\R\to \R$ satisfying
\begin{equation*}
\mbox{$\chi=1$ on $[-1,1]$, $\chi=0$ on $(-\infty,-2]\cup[2,+\infty)$,
$\chi'\leq 0$ on $[0,+\infty)$.}
\end{equation*}
The constants $A$ and $B$ will be fixed later such that
$1\ll B \ll A$.
We define the following functions on $\R$
\begin{equation*}
\zeta_A(x)=\exp\left(-\frac1A(1-\chi(x))|x|\right),\quad
\Phi_A(x)=\int_0^x \zeta_A^2(y) \ud y
\end{equation*}
and
\begin{align}
&\zeta_B(x)=\exp\left(-\frac1B(1-\chi(x))|x|\right),\quad
\Phi_B(x)=\int_{0}^x \zeta_B^2(y) \ud y,\nonumber\\
&\Psi_{A,B}(x)=\chi_A^2(x)\Phi_B(x)\quad\mbox{where}\quad
\chi_A(x)=\chi\left(\frac{x}{A}\right).\label{def:chiB}
\end{align}
We also set
\[
\sigma_A(x) = \sech\left( \frac 2A x\right).
\]

\subsubsection*{General virial computation}

Let $\Phi:\R\to\R$ be a smooth, odd, strictly increasing, bounded function and let $\boldsymbol a=(a_1, a_2)$ be an $H^1\times L^2$ solution of 
\[
\begin{cases}
\dot a_1= a_2 \\
\dot a_2=-L a_1+G
\end{cases}
\] 
where $L=-\partial_{xx}+ P
$
and $G=G(t,x)$, $P=P(x)$ are given functions.
Define
\begin{equation*}
\cA=\int\left(\Phi \partial_x a_1+\frac{1}{2}\Phi' a_1\right)a_2
\end{equation*}
We compute formally
\begin{equation}\label{eq:vir0}
\dot \cA=-\int\Phi' (\partial_x a_1)^2+ \frac{1}{4}\int \Phi''' a_1^2 +\frac{1}{2} \int \Phi P'a_1^2
+\int G\left(\Phi \px a_1+\frac{1}{2}\Phi' a_1\right).
\end{equation}
Changing variables
\[
\tilde a_1=\zeta a_1, \qquad \zeta=\sqrt{\Phi'},
\]
we recall the alternative identity
\begin{equation}
\label{eq:vir1}
\begin{aligned}
\dot \cA &=-\int (\partial_x \tilde a_1)^2-\frac{1}{2} \int \left(\frac{\zeta''}{\zeta} - \frac{(\zeta')^2}{\zeta^2}\right) \tilde a_1^2
+\frac 12 \int \frac{\Phi P'}{\zeta^2} \tilde a_1^2
+\int G\left(\Phi \px a_1+\frac{1}{2}\Phi' a_1\right).
\end{aligned}
\end{equation}
We refer to \cite[Proof of Proposition 1]{KMM4} for detailed computations.

\section{First estimate at large scale}
The first key estimate of the proof of Theorem~\ref{TH:1}, given in the next proposition, controls (in time average) norms of $(u_1,u_2)$ at any large scale $A$ in terms of $|\zz|^2$ and a local norm of $u_1$.
The proof relies on a virial estimate for the system~\eqref{eq:zu}, without any spectral information.

\begin{proposition}\label{PR:1}
For any $A>0$ large, any $\delta>0$ small (depending on $A$) and any $T>0$, it holds
\[
\int_0^T \left( \|\sigma_A \partial_x u_1\|^2 + \frac 1{A^2} \|\sigma_A u_1\|^2 
+ \frac 1{A^2} \|\sigma_A u_2\|^2 \right) \udd t
\lesssim A \delta^2 + \int_0^T \left(\|\rho^2 u_1 \|^2+ |\zz|^4 \right) \udd t .
\]
\end{proposition}

\begin{proof}
Consider the time dependent function $\cI(t)$ defined by
\[
\cI = \int\left(\Phi_A \partial_x u_{1}+\frac{1}{2}\Phi_A' u_1\right) u_2
\]
and denote $\tilde u_1=\zeta_A u_1$.
The system \eqref{eq:zu} and the identity \eqref{eq:vir1} imply
\begin{equation}
\begin{aligned}\label{est:vir1}
\dot \cI 
&= -\int (\partial_x \tilde u_1)^2 
- \frac{1}{2} \int \left(\frac{\zeta_A''}{\zeta_A}- \frac{(\zeta_A')^2}{\zeta_A^2}\right) \tilde u_1^2 
+ \frac 12 \int \frac{\Phi_A}{\zeta_A^2} [W''(H)]' \tilde u_1^2 \\
& \quad - \int N^\perp \left(\Phi_A \partial_xu_1 +\frac{1}{2}\Phi'_A u_1\right).
\end{aligned}
\end{equation}
First, by elementary computations (see detailed computations in \cite[Proof of Lemma 1]{KMM4})
\begin{equation*}
\frac{\zeta_A''}{\zeta_A}-\frac{(\zeta_A')^2}{\zeta_A^2}
=\frac 1A\left[\chi''(x) |x|+2\chi'(x)\sgn(x)\right]
\end{equation*}
and so  
\begin{equation*}
\left|\frac{\zeta_A''}{\zeta_A}-\frac{(\zeta_A')^2}{\zeta_A^2}\right|
\lesssim \frac{\rho^4}A.
\end{equation*}
This implies that
\[
\left|\int \left(\frac{\zeta_A''}{\zeta_A} - \frac{(\zeta_A')^2}{\zeta_A^2}\right) \tilde u_1^2\right|
\lesssim \frac 1A \|\rho^2 \tilde u_1\|^2 \lesssim \frac 1A \|\rho^2 u_1\|^2.
\]
Second, since 
$[W''(H)]'=H' W'''(H)\in \cY$, $\zeta_A\geq e^{-\frac{|x|}A}$ and 
$|\Phi_A|\leq |x|$, we also have, for $A$ large,
\[
\left|\frac{\Phi_A}{\zeta_A^2} [W''(H)]'\right|
\lesssim |x| e^{\frac2A|x|} \rho^5 \lesssim \rho^4.
\]
Thus,
\[
\left|\int \frac{\Phi_A}{\zeta_A^2} [W''(H)]' \tilde u_1^2\right|
\lesssim \|\rho^2 \tilde u_1\|^2 \lesssim \|\rho^2 u_1\|^2.
\]
For the last term in \eqref{est:vir1}, we decompose $N$ defined in \eqref{def:N}, as
\begin{equation*}
N=N_1+N_2
\end{equation*}
where
\begin{align*}
N_1 & = W'(H+\varphi_1)-W'(H+u_1)-W''(H)(\varphi_1-u_1),\\
N_2 & = W'(H+u_1)-W'(H)-W''(H)u_1.
\end{align*}
Using $\varphi_1=u_1 + z_1 Y$, we have by Taylor expansion and $Y\in \cY$,
\[
|N_1|
= |W'(H+u_1+z_1Y)-W'(H+u_1)- W''(H)z_1 Y| \lesssim z_1^2 Y^2
\lesssim z_1^2 \rho^{10}.
\]
Thus, using $|\Phi_A|\leq |x|$, $|\Phi_A'|\leq 1$ and the Cauchy Schwarz inequality,
\[
\left| \int N_1^\perp \left(\Phi_A \partial_x u_1+\frac{1}{2}\Phi'_A u_1\right) \right|
\lesssim z_1^2 \left( \|\rho^2 \partial_x u_1\| + \|\rho^2 u_1\|\right).
\]
To deal with the term $N_2$, we set
\begin{align*}
G_1 &= W(H+u_1)-W(H)-W'(H)u_1-\frac 12 W''(H)u_1^2,\\
G_2 &= W'(H+u_1)-W'(H)-W''(H) u_1- \frac 12 W'''(H) u_1^2
\end{align*}
so that
$\partial_x G_1 = N_2 \partial_x u_1 + G_2 H'$ and thus
\[
N_2^\perp \partial_x u_1 = \partial_x G_1 - G_2 H' - {\langle N_2,Y\rangle} Y \partial_x u_1.
\]
Therefore, by integration by parts,
\begin{align*}
\int N_2^\perp \left(\Phi_A \partial_x u_1+\frac{1}{2}\Phi'_A u_1\right)
&= \int \left(-G_1+\frac 12 N_2 u_1\right) \Phi_A' - \int G_2H' \Phi_A \\
&\quad + \langle N_2,Y\rangle \int \left( \Phi_A Y' + \frac 12 \Phi_A' Y\right) u_1.
\end{align*}
Using the following pointwise estimates (Taylor expansions)
\begin{align*}
|G_1|+|G_2| \lesssim |u_1|^3, \quad
|N_2|\lesssim |u_1|^2,
\end{align*}
and $H'$, $Y\in \cY$, $|\Phi_A'|\leq 1$, $|\Phi_A|\lesssim |x|$, we have
\begin{align*}
&\left|\int G_1 \Phi_A' \right|+\left|\int N_2 u_1\Phi_A' \right|\lesssim \int \Phi_A' |u_1|^3,\\
&\left|\langle N_2,Y\rangle \int \left( \Phi_A Y' + \frac 12 \Phi_A' Y\right) u_1\right|
\lesssim \|\rho^2 u_1\|^3 \lesssim \delta \|\rho^2 u_1\|^2.
\end{align*}
Thus,
\[
\left| \int N_2^\perp \left(\Phi_A \partial_x u_1+\frac{1}{2}\Phi'_A u_1\right)\right| \lesssim 
\int \Phi_A' |u_1|^3 +\delta\|\rho^2 u_1\|^2.
\]
Using $\Phi_A'=\zeta_A^2$ and the following estimate from \cite[Claim 1]{KMM4}
\begin{equation*}
\int \zeta_A^2 |u_1|^{3}\lesssim A^2\|u_1\|_{L^\infty}\|\partial_x \tilde u_1\|^2
\end{equation*}
we obtain using~\eqref{eq:unif},
\begin{align*}
\left| \int N_2^\perp \left(\Phi_A \partial_x u_1+\frac{1}{2}\Phi'_A u_1\right)\right| 
&\lesssim A^2\|u_1\|_{L^\infty} \|\partial_x \tilde u_1\|^2 +\delta\|\rho^2 u_1\|^2\\
&\lesssim A^2\delta \|\partial_x \tilde u_1\|^2 +\delta\|\rho^2 u_1\|^2 .
\end{align*}
Taking $\delta$ small depending on $A$, we find, for a constant $C>0$,
\[
\dot \cI \leq - \frac 12 \|\partial_x \tilde u_1\|^2 + C \|\rho^2 u_1\|^2
+C z_1^2 \left( \|\rho^2 \partial_x u_1\| + \|\rho^2 u_1\|\right).
\]

Now, we claim 
\begin{equation}\label{eq:claim}
\|\sigma_A \partial_x u_1\|^2 + \frac 1{A^2} \|\sigma_A u_1\|^2
\lesssim \|\partial_x \tilde u_1\|^2
+ \frac 1A\|\rho^2 u_1\|^2.
\end{equation}
Indeed, first, we compute using $\tilde u_1 = \zeta_A u_1$ and integration by parts,
\[
\int \zeta_A^2 |\partial_x \tilde u_1|^2 = \int \zeta_A^4 |\partial_x u_1|^2 - \int \zeta_A^3\zeta_A'' u_1^2
- 2 \int \zeta_A^2(\zeta_A')^2 u_1^2,
\]
so that using $\sigma_A \lesssim \zeta_A^2 \lesssim \sigma_A$ (by definition of $\zeta_A$ and $\sigma_A$) and $|\zeta_A''|\zeta_A+|\zeta_A'|^2\leq A^{-2} \sigma_A$, we have 
\[
\int \sigma_A^2 |\partial_x u_1|^2 \lesssim \int \sigma_A |\partial_x \tilde u_1|^2 + A^{-2} \int \sigma_A\tilde u_1^2.
\]
Second, we recall the elementary inequality (see \emph{e.g.}~\cite[Lemma~4 (2)]{KMM4})
\[
 \int \sigma_A |\tilde u_1|^2
\lesssim A^2 \int |\partial_x \tilde u_1|^2 + A \int \rho^4 \tilde u_1^2.
\]
This is sufficient to complete the proof of \eqref{eq:claim}.

We conclude, for constants $C$, $C'>0$,
\[
\|\sigma_A \partial_x u_1\|^2 + \frac 1{A^2} \|\sigma_A u_1\|^2\leq - C \dot \cI + C' \left(\|\rho^2 u_1\|^2 + |\zz|^4\right).
\]
By integrating this estimate on $[0,T]$, using $|\cI|\leq C A \left( \|u_1\|_{H^1}^2+ \|u_2\|^2\right)
\leq C A \delta^2$,
we find
\[
\int_0^T \left( \|\sigma_A \partial_x u_1\|^2 + \frac 1{A^2} \|\sigma_A u_1\|^2 \right) \udd t
\lesssim A \delta^2 + \int_0^T \left(\|\rho^2 u_1 \|^2+ |\zz|^4 \right) \udd t .
\]

To complete the proof of the proposition, we only need to prove the same estimate for $\int_0^T \|\sigma_A u_2\|^2\ud t$.
For this, we introduce the functional
\[
\cH = \int \sigma_A^2 u_1 u_2 .
\]
Using~\eqref{eq:zu}, the expression of $L_0$ and integration by parts, we compute
\[
\dot \cH =
\int \sigma_A^2 u_2^2 - \int \sigma_A^2 (\partial_x u_1)^2 + \frac 12 \int (\sigma_A^2)'' u_1^2
- \int \sigma_A^2 W''(H) u_1^2 - \int \sigma_A^2 N^\perp u_1.
\]
Using $|(\sigma_A^2)''|\lesssim A^{-2} \sigma_A^2\lesssim \sigma_A^2$, $|W''(H)|\lesssim 1$, and 
$|N^\perp|\lesssim u_1^2 + Y^2 z_1^2$, we find for some constant $C>0$,
\[
\dot \cH \geq 
\|\sigma_A u_2\|^2 - \|\sigma_A \partial_x u_1\|^2 - C\left(\|\sigma_A u_1\|^2+|\zz|^4 \right) .
\]
By integrating on $[0,T]$, using $|\cH|\lesssim\|u_1\|\|u_2\|\lesssim \delta^2$,
we obtain
\begin{equation*}
\frac 1{A^2}\int_0^T \|\sigma_A u_2\|^2 \ud t
\lesssim \delta^2+\frac 1{A^2} \int_0^T \left( \|\sigma_A \partial_x u_1\|^2 + \|\sigma_A u_1\|^2+|\zz|^4 \right) \udd t,
\end{equation*}
which implies the desired estimate.
\end{proof}

\section{Estimate of the internal mode component}

The second key estimate of the proof of Theorem~\ref{TH:1}, given in the next proposition, controls (in time average) the function $|\zz|^2$ in terms of a local norm of $u_1$.
Its proof relies on the Fermi golden rule (Hypothesis 2) and the introduction of a new functional $\mathcal J$ (see Remark~\ref{RK:2} after the proof for a comparison with the argument in \cite{KMM}).

\begin{proposition}\label{PR:2}
For any $A>0$ large, any $\delta>0$ small (depending on $A$) and any $T>0$, it holds\[
\int_0^T |\zz|^4 \ud t
\lesssim A \delta^2 + \frac 1{\sqrt{A}} \int_0^T \|\rho^2 u_1 \|^2 \ud t .
\]
\end{proposition}
\begin{proof}
We introduce the time dependent functions
\[
\alpha = z_1^2-z_2^2, \qquad \beta = 2z_1z_2.
\]
Using \eqref{eq:zu}, we have
\begin{equation}\label{eq:ab}
\begin{cases}
\dot \alpha = 2 \lambda \beta + \frac2\lambda\langle N,Y\rangle z_2 \\
\dot \beta = -2 \lambda \alpha - \frac2\lambda \langle N,Y\rangle z_1 .
\end{cases}
\end{equation}
Moreover,
\begin{equation}\label{eq:z2}
\frac \ud{\ud t} \left( |\zz|^2 \right) = -\frac{2}\lambda \langle N,Y\rangle z_2 .
\end{equation}
By Hypothesis 2, there exists a $\cC^\infty$ bounded function $g$ satisfying
$L_0 g = 4 \lambda^2 g$ and such that
\[
\Gamma := \frac 14 \int W'''(H) Y^2 g \neq 0.
\]
We define the functional $\mathcal J$ by
\[
\mathcal J = - \alpha \int u_2 g \chi_A + 2 \lambda \beta \int u_1 g \chi_A + \frac {\Gamma}{2\lambda}\beta |\zz|^2 
\] 
where $\chi_A$ is defined in \eqref{def:chiB}.
By direct computations, using \eqref{eq:zu}, \eqref{eq:ab} and \eqref{eq:z2}, 
\begin{align*}
\dot{\mathcal J} 
&= \alpha \int \left(L_0 u_1 - 4\lambda^2 u_1\right) g\chi_A\\
&\quad + \alpha \left( - \Gamma |\zz|^2 + \int N^\perp g \chi_A \right)\\
&\quad - \frac2\lambda \langle N,Y\rangle \left(z_2 \int u_2 g \chi_A + 2 \lambda z_1 \int u_1 g \chi_A\right) \\
&\quad +\frac{\Gamma}{\lambda^2} \langle N,Y\rangle \left(z_1 |\zz|^2 + z_2 \beta\right) 
= J_1 + J_2 + J_3 + J_4.
\end{align*}
For $J_1$, we integrate by parts 
\[
\int u_1 g'' \chi_A 
= \int (\partial_{xx} u_1) g \chi_A + 2 \int (\partial_x u_1) g \chi_A' + \int u_1 g \chi_A'',
\]
and use $(L_0 - 4\lambda^2) g = 0$ to obtain
\[
J_1 = -\alpha \left( 2 \int (\partial_x u_1) g \chi_A' + \int u_1 g \chi_A''\right).
\]
For the first term,
using $|g|\lesssim 1$, the definition of $\chi_A$, the Cauchy-Schwarz inequality and then $\sigma_A \gtrsim 1$ on $[-2A,2A]$, we have 
\begin{align*}
\left| \int (\partial_x u_1) g \chi_A' \right|
&\lesssim \frac 1A\int_{|x|<2A} |\partial_x u_1|
\lesssim \frac 1{\sqrt{A}} \left(\int_{|x|<2A} |\partial_x u_1|^2\right)^{\frac 12}
\lesssim \frac 1{\sqrt{A}} \|\sigma_A \partial_x u_1\|.
\end{align*}
Similarly, for the second term in $J_1$, we have
\[
\left|\int u_1 g \chi_A''\right|
\lesssim \frac 1{A^2} \int_{|x|<2A} |u_1|
\lesssim \frac 1{A\sqrt{A}} \left(\int_{|x|<2A} u_1^2 \right)^{\frac 12}
\lesssim \frac 1{A\sqrt{A}} \| \sigma_A u_1\|.
\]
Therefore, for $J_1$, we have
\[
|J_1|\lesssim 
\frac {|\alpha|}{\sqrt{A}}\left(\|\sigma_A \partial_x u_1\|^2 + \frac 1{A^2} \|\sigma_A u_1\|^2\right)^{\frac 12}
\lesssim \frac 1{\sqrt{A}} \left(\|\sigma_A \partial_x u_1\|^2 + \frac 1{A^2} \|\sigma_A u_1\|^2
+|\zz|^4\right).
\]
We turn to the term $J_2$.
We decompose from \eqref{def:N},
\begin{equation}\label{eq:decN}
N= z_1^2 R_0 + R_1+R_2
\end{equation}
where
\begin{align*}
R_0&= \frac 12 W'''(H) Y^2,\\
R_1&=z_1 W'''(H)Y u_1 + \frac 12 W'''(H) u_1^2,\\
R_2&= W'(H+\varphi_1)-W'(H)-W''(H)\varphi_1
-\frac 12 W'''(H) \varphi_1^2.
\end{align*}
We compute using $\int R_0 g = 2 \Gamma$,
\begin{align*}
\int N^\perp g \chi_A & =
\int N g \chi_A - \langle N,Y\rangle \int Y g \chi_A\\
& = 2\Gamma z_1^2 - z_1^2\int R_0 g (1-\chi_A) +\int R_1 g \chi_A + \int R_2 g \chi_A- \langle N,Y\rangle \int Y g \chi_A.
\end{align*}
By $|g|\lesssim 1$, $R_0\in \cY$ and the definition of $\chi_A$, we have
\[
\left|z_1^2 \int R_0 g (1-\chi_A) \right| \lesssim z_1^2 \int_{|x|\geq A} \rho^{10}
\lesssim \rho(A)z_1^2.
\]
By the definition of $R_1$ and $Y\in \mathcal Y$, we have
$|R_1|\lesssim |z_1| |u_1| \rho^5 + u_1^2$. Thus, 
using $\chi_A \lesssim \sigma_A^2$,
\[
\left| \int R_1 g \chi_A\right|
\lesssim |z_1| \|\rho^2 u_1\| + \|\sigma_A u_1\|^2.
\]
By the definition of $R_2$ and $Y\in \mathcal Y$,
\[
|R_2|\lesssim |\varphi_1|^3 \lesssim |Y|^3 |z_1|^3 + |u_1|^3
\lesssim \delta \left(\rho^{15} z_1^2 + |u_1|^2\right).
\]
Thus, using $\chi_A \lesssim \sigma_A^2$,
\[
\left| \int R_2 g \chi_A\right|
\lesssim \delta \left(z_1^2 + \|\sigma_A u_1\|^2\right).
\]
Note that since $L_0Y=\lambda^2 Y$ and $L_0 g = 4\lambda^2 g$, 
we have $\int Y g = 0$.
In particular, since $Y\in \mathcal Y$,
\[
\left| \int Y g \chi_A \right| = \left| \int Y g (1 - \chi_A) \right| 
\lesssim \int_{|x|>A} \rho^5\lesssim \rho(A).
\]
Moreover, by the definition of $N$, we have
\begin{equation}\label{NY}
| \langle N, Y \rangle | \lesssim |z_1|^2 + \|\rho^2 u_1\|^2.
\end{equation}
Thus,
\[
\left| \langle N,Y\rangle \int Y g \chi_A\right|
\lesssim \rho(A) \left(z_1^2 + \|\rho^2 u_1\|^2\right).
\]
Gathering the last estimates, we have proved
\[
\left|2 \Gamma z_1^2 - \int N^\perp g \chi_A \right|
\lesssim |z_1| \|\rho^2 u_1\| + \|\sigma_A u_1\|^2 + (\rho(A)+\delta) z_1^2 . 
\]
Since $2z_1^2-|\zz|^2 = \alpha$,
we have obtained
\begin{align*}
\left|J_2 - \Gamma \alpha^2 \right|
&\lesssim |\alpha| \left( |z_1| \|\rho^2 u_1\| + \|\sigma_A u_1\|^2 + (\rho(A)+\delta) z_1^2 \right)\\
&\lesssim \left(\frac1{\sqrt{A}}+\delta\right) |\zz|^4+ \delta \|\sigma_A u_1\|^2.
\end{align*}
Next, we estimate $J_3$. Using \eqref{NY} and the Cauchy Schwarz inequality, we have
\begin{align*}
|J_3|
&\lesssim |\zz| | \langle N, Y \rangle | \sqrt{A} \left( \|\sigma_A u_1\| + \|\sigma_A u_2\| \right)\\
&\lesssim |\zz| \sqrt{A} \left( |z_1|^2 + \|\rho^2 u_1\|^2 \right) \left( \|\sigma_A u_1\| + \|\sigma_A u_2\| \right)\\
&\lesssim \delta |\zz|^4+ \delta A\left(\|\sigma_A u_1\|^2 + \|\sigma_A u_2\|^2 \right).
\end{align*}
Last, we have similarly
\[
|J_4| \lesssim |\zz|^3 | \langle N, Y \rangle |
\lesssim |\zz|^3\left( |z_1|^2 + \|\rho u_1\|^2 \right)
\lesssim \delta |\zz|^4+ \delta\|\sigma_A u_1\|^2 .
\]
In conclusion, choosing $ \delta\leq A^{-\frac 72}$, we have obtained
\begin{equation}\label{eq:dJ}
\left| \dot {\mathcal J} - \Gamma \alpha^2 \right| 
\lesssim \frac 1{\sqrt{A}} \left(\|\sigma_A \partial_x u_1\|^2 + \frac 1{A^2} \|\sigma_A u_1\|^2
 + \frac 1{A^2} \|\sigma_A u_2\|^2+|\zz|^4\right). 
\end{equation}

Next, we set
\[
\cZ = \frac{\Gamma}{4\lambda}\alpha \beta.
\]
Using \eqref{eq:ab}, we have
\[
\dot \cZ = \frac \Gamma 2 \left(\beta^2 - \alpha^2\right)
+ \frac{\Gamma}{2\lambda^2} \langle N, Y \rangle \left( \beta z_2 - \alpha z_1\right).
\]
Since $|\langle N, Y \rangle \left( \alpha z_1 +\beta z_2 \right)|\lesssim |\zz|^3 |\langle N,Y\rangle|$, we estimate this term as the term $J_4$ above. We obtain, for $A$ large depending on $\delta$,
\begin{equation}\label{eq:dZ}
\left| \dot\cZ - \frac \Gamma 2\left(\beta^2 - \alpha^2\right)\right| 
\lesssim \frac 1{\sqrt{A}} \left(\|\sigma_A \partial_x u_1\|^2 + \frac 1{A^2} \|\sigma_A u_1\|^2
 + \frac 1{A^2} \|\sigma_A u_2\|^2+|\zz|^4\right).
\end{equation}

Combining \eqref{eq:dJ}, \eqref{eq:dZ} and observing that $\alpha^2 + \beta^2 = |\zz|^4$, we obtain
\begin{equation*}
\left| \dot {\mathcal J} + \dot \cZ- \frac{\Gamma}{2} |\zz|^4 \right| 
\lesssim \frac 1{\sqrt{A}} \left(\|\sigma_A \partial_x u_1\|^2 + \frac 1{A^2} \|\sigma_A u_1\|^2
 + \frac 1{A^2} \|\sigma_A u_2\|^2+|\zz|^4\right). 
\end{equation*}
In particular, for a constant $C$,
\[
|\zz|^4 
\leq \frac 2\Gamma \left( \dot {\mathcal J} + \dot \cZ\right)
+\frac C{\sqrt{A}} \left(\|\sigma_A \partial_x u_1\|^2 + \frac 1{A^2} \|\sigma_A u_1\|^2
+ \frac 1{A^2} \|\sigma_A u_2\|^2+|\zz|^4\right).
\]
By integrating this estimate on $[0,T]$, using
\[
|\cJ| 
\lesssim |\zz|^4+|\zz|^2\int_{|x|\leq 2A} \left(|u_1|+|u_2| \right)
\lesssim |\zz|^4 + |\zz|^2 \sqrt{A} \left(\|u_1\|+\|u_2\|\right)
\lesssim \sqrt{A}\delta^3,
\]
and $|\cZ|\lesssim |\zz|^4 \lesssim \delta^4$,
we conclude
\[
\int_0^T |\zz|^4 \ud t 
\lesssim \sqrt{A} \delta^4 + \frac 1{\sqrt{A}} \int_0^T \left(\|\sigma_A \partial_x u_1\|^2 + \frac 1{A^2} \|\sigma_A u_1\|^2
+ \frac 1{A^2} \|\sigma_A u_2\|^2+|\zz|^4\right) \udd t.
\]
Last, using 
Proposition \ref{PR:1}, we obtain
\[
\int_0^T |\zz|^4 \ud t 
\lesssim \sqrt{A} \delta^2 + \frac 1{\sqrt{A}} \int_0^T \left(\|\rho^2 u_1\|^2 + |\zz|^4\right) \udd t.
\]
Taking $A$ large enough (independent of $\delta$), we obtain
the estimate of Proposition~\ref{PR:2}.
\end{proof}

\begin{remark}\label{RK:2}
The functional $\mathcal J$ introduced in the proof of Proposition~\ref{PR:2}
may seem similar to the one introduced in \cite[Eq. (4.7)]{KMM}, however we observe that the use of the function $g$ in the present paper allows to rely on the natural Fermi golden rule (Hypothesis 2)
while one of the conditions checked numerically in \cite{KMM} is a perturbation  of the Fermi golden rule.
\end{remark}

\section{The Darboux factorization}

First, we prove bounds on the operator $S_\varepsilon=\Xe U_1 U_0$. Second, we prove a coercivity property on $S_\varepsilon$.

\begin{lemma}\label{LE:1}
For any $A>0$ large, any~$\varepsilon>0$ small and any $u\in H^1$,
\begin{align*}
\| \sigma_A S_\varepsilon u \| & \lesssim \varepsilon^{-1} \|\sigma_A u\| ,\\
\| \sigma_A \px S_\varepsilon u\| & \lesssim \varepsilon^{-1} \|\sigma_A \px u\|+ \|\rho^2 u\| .
\end{align*}
\end{lemma}
\begin{proof}
By direct computations, we have
\[
U_1U_0=\partial_{xx} - \partial_x \cdot k_1 + k_2
.
\]
where the functions $k_1$ and $k_2$ are defined by
\[
k_1= \frac{Z'}{Z} + \frac{H''}{H'} ,\quad k_2= \left(\frac{Z'}{Z}\right)'+\frac{Z'}{Z}\frac{H''}{H'}.
\]
Observe that $k_1,k_2$ are bounded.
Thus, the first estimate is a consequence the following technical estimates from \cite[Lemma~4.7]{KMMV}:
 for $\varepsilon>0$ small and
any $h\in L^2$,
\begin{equation}\label{tech1}
\|\sigma_A \Xe h\| \lesssim \| \sigma_A h\|,
\quad \|\sigma_A \Xe \px h\| \lesssim \varepsilon^{-\frac 12}\|\sigma_A h\|,
\quad \|\sigma_A \Xe \partial_{xx}h\| \lesssim \varepsilon^{-1}\|\sigma_A h\|.
\end{equation}
Besides,
\[
\px U_1U_0=\partial_{xxx} - \partial_x (k_1 \partial_x) + (k_2-k_1') \partial_x + k_2'-k_1''.
\]
Since $k_1$, $k_2$ are bounded and $k_1'$, $k_2'\in\cY$, we obtain the second estimate
using \eqref{tech1}.
\end{proof}

\begin{lemma}\label{LE:2}
For any odd function $u\in L^2$ such that $\langle u, Y\rangle=0$, the following estimate holds 
\begin{equation*}
\|\rho^2 u \| \lesssim \|\rho S_\varepsilon u\| .
\end{equation*} 
\end{lemma}
\begin{proof}
Let $u$ be an odd function in $L^2$ such that $\langle u, Y\rangle=0$. Note that this condition is necessary for the coercivity because of~\eqref{eq:equiv}.
Denote $v=S_\varepsilon u = \Xe U_1 U_0 u$.
We have
\[
v-\varepsilon\partial_{xx} v= U_1 U_0 u.
\]
Setting
\[
w=\left[ 1 - \varepsilon Z \left(\frac 1Z\right)'' \right] v,
\]
and using $U_1=Z \cdot \partial_x \cdot Z^{-1}$, we check that the above relation can be written 
\[
\partial_x\left(\frac {U_0 u}Z \right) = -\varepsilon \px\left(\frac{\partial_x v}Z + \frac{Z'}{Z^2} v \right)+ \frac{w}{Z}.
\]
Integrating and rearranging we obtain, denoting by $a$ the constant of integration
\[
U_0 u + \varepsilon \partial_x v = a Z -\varepsilon \frac{Z'}Z v + Z\int_0^x \frac{w}{Z}.
\]
Using $U_0=H' \cdot \partial_x \cdot (H')^{-1}$ and $Z=U_0Y$, the above relation can be written
\[
\partial_x \left[ \frac1{H'} \left(u +\varepsilon v\right)\right]
= a \partial_x \left(\frac{Y}{H'}\right) -\varepsilon \left( \frac{Z'}{H'Z}+\frac{H''}{(H')^2}\right) v + \frac{Z}{H'}\int_0^x \frac{w}{Z}. 
\]
Integrating once more we find, denoting by $b$ the constant of integration
\begin{equation}
\label{est:y1}
u= -\varepsilon v +a Y +b H'+ H'\int_0^x \left[ -\frac{\varepsilon}{H'} \left( \frac{Z'}{Z}+\frac{H''}{H'}\right) v 
+ \frac{Z}{H'}\int_0^y \frac{w}{Z}\right] .
\end{equation}
Since $H'$ and $Z$ are even and $u$, $v$, $w$, $Y$, are odd, we have $b=0$ by parity.
Besides, since (by standard ODE arguments)
\[
\left| \frac{Z'}{Z}\right|+\left|\frac{H''}{H'}\right| \lesssim 1,
\]
by the Cauchy-Schwarz inequality
\[
\left| H'\int_0^x \frac{1}{H'} \left( \frac{Z'}{Z}+\frac{H''}{H'}\right) v\right|
\lesssim H' \int_0^x \frac{|v|}{H'} 
\lesssim H' \left( \int_0^x \frac 1{(H'\rho)^2}\right)^\frac 12 \|\rho v\|
\lesssim \frac{\|\rho v\|}{\rho}.
\]
Similarly, using also $|w|\lesssim |v|$ since $\big|Z\big(\frac 1Z\big)''\big|\lesssim 1$, we have
\[
\left| \frac{Z}{H'}\int_0^y \frac{w}{Z} \right| 
\lesssim \frac{\|\rho w\|}{H'\rho}\lesssim \frac{\|\rho v\|}{H'\rho},
\]
and thus
\[
\left| H'\int_0^x \left[ \frac{Z}{H'} \int_0^y \frac{w}{Z} \right]\right| 
\lesssim \frac{\|\rho v\|}{\rho}.
\]
Multiplying \eqref{est:y1} by $Y\in \cY$ and integrating, using $\langle u,Y\rangle =0$, we find
$|a|\lesssim \|\rho v\|$.
Gathering the above estimates, we have proved
\[
|u| \lesssim |v| + \frac {\|\rho v\|}{\rho},
\]
which directly implies that $\|\rho^2 u\|\lesssim \|\rho v\|$.
\end{proof}
 
\section{Virial estimate for the transformed problem}

The last key estimate of the proof of Theorem~\ref{TH:1} 
will allow us to close  the estimates relating $\uu$, $\zz$ and $S_\varepsilon u_1$.
Its proof relies on the spectral assumption (Hypothesis~3) and
a virial computation on the transformed problem satisfied by $S_\varepsilon \uu$.

\begin{proposition}\label{PR:3}
For any $\varepsilon>0$ small, any $A>0$ large, any $\delta>0$ small (depending on $\varepsilon$ and $A$) and any $T>0$,
\begin{equation*}
\int_0^T \|\rho S_\varepsilon u_1\|^2 \ud t 
\lesssim A \delta^2 +\frac 1{\sqrt{A}} \int_0^T \|\rho^2 u_1\|^2 \ud t .
\end{equation*}
\end{proposition}
\begin{proof}
We set
$v_1 = S_\varepsilon u_1$ and
$v_2 = S_\varepsilon u_2$.
From \eqref{eq:zu} and the identity \eqref{eq:id}, we check that
\begin{equation}
\label{eq:v}
\begin{cases}
\dot v_1= v_2\\
\dot v_2=-L_2 v_1 - \left[\Xe,P_2\right] U_1 U_0 u_1 - S_\varepsilon N^\perp
\end{cases}
\end{equation}
where we denote $[\Xe,P_2] = \Xe P_2 -P_2 \Xe$.

The function $\Psi_{A,B}$ being defined in \eqref{def:chiB}, we set 
\[
\cK =\int\left(\Psi_{A,B} \partial_x v_{1}+\frac{1}{2}\Psi_{A,B}' v_1\right)v_2.
\]
Taking the time derivative of $\cK$, using~\eqref{eq:v} and the general computation~\eqref{eq:vir0}, we find
\begin{equation}
\label{virial2}
\begin{aligned}
\dot\cK 
&=-\int\Psi_{A,B}'(\partial_x v_1)^2+\frac{1}{4}\int \Psi_{A,B}''' v_1^2 +\frac{1}{2} \int \Psi_{A,B} P_2'v_1^2\\
&\quad -\int \left(\Psi_{A,B} \partial_x v_{1}+\frac{1}{2}\Psi_{A,B}' v_1\right)\left[\Xe,P_2\right] U_1 U_0u_1\\
&\quad -\int \left(\Psi_{A,B} \partial_x v_{1}+\frac{1}{2}\Psi_{A,B}' v_1\right)S_\varepsilon N^\perp\\
&=K_1+K_2+K_3.
\end{aligned}
\end{equation}
We denote $\tilde v_1 = \chi_A\zeta_B v_1$.
Following the calculations in \cite[Sect. 4.3]{KMM4},
we check that
\begin{equation*}
K_1
 =- \int \left[ (\partial_x \tilde v_1)^2 +V_B \tilde v_1^2\right]+\widetilde K_1
\end{equation*}
where
\begin{equation*}
V_B=\frac 12 \left( \frac{ \zeta_B''}{\zeta_B}- \frac{ (\zeta_B')^2}{\zeta_B^2} \right) 
-\frac 12 \frac{\Phi_B}{ \zeta_B^2 }P_2'
\end{equation*}
and
\begin{align*} 
\widetilde K_1
&= \frac14 \int (\chi_A^2)' (\zeta_B^2)' v_1^2 +\frac12 \int \left[3(\chi_A')^2 + \chi_A'' \chi_A\right]\zeta_B^2 v_1^2\\
&\quad - \int (\chi_A^2)' \Phi_B (\partial_x v_1)^2+\frac14 \int (\chi_A^2)''' \Phi_B v_1^2.
\end{align*}
In the next lemma, we prove a lower bound on the quantity $\int \left[ (\partial_x \tilde v_1)^2 +V_B \tilde v_1^2\right]$ using Hypothesis 3.
\begin{lemma}\label{LE:3}
There exist $B_0>0$ and $\mu>0$, such that for all $B\geq B_0$, it holds
\begin{equation}\label{on:VB}
\int \left[ (\partial_x \tilde v_1)^2 +V_B \tilde v_1^2\right]
\geq \mu \left( \|\rho \px v_1 \|^2+\|\rho v_1 \|^2\right)
- \frac 1\mu \frac 1A \left( \|\sigma_A \px u_1\|^2 + \frac 1{A^2} \|\sigma_A u_1\|^2\right).
\end{equation}
\end{lemma}
\begin{proof}
We claim that for some $\tilde \mu>0$,
\begin{equation}\label{old:LE5}
\int \left[ (\partial_x \tilde v_1)^2 +V_B \tilde v_1^2\right]
\geq \tilde\mu \int \rho \left[ (\px \tilde v_1)^2+\tilde v_1^2 \right].
\end{equation}
First, as in the proof of Proposition~\ref{PR:2},
\[
\left|\frac{\zeta_B''}{\zeta_B}-\frac{(\zeta_B')^2}{\zeta_B^2}\right|
\lesssim \frac{\rho^2}B.
\]
Next, the function $\tilde v_1$ being odd,
 by Hypothesis~3 there exists $\gamma>0$ such that
\[
(1-\gamma) \int (\partial_{x}\tilde v_1)^2 \geq -\frac 12 \int xP_2' \tilde v_1^2.
\]
Thus, for a constant $C>0$,
\[
\int \left[ (\partial_x \tilde v_1)^2 +V_B \tilde v_1^2\right]
\geq \gamma \int (\partial_x \tilde v_1)^2- \frac CB \int \tilde v_1^2 \rho^2 - \frac 12\int \left|x - \frac{\Phi_B}{\zeta_B^2}\right| |P_2'| \tilde v_1^2.
\]
We claim the following pointwise estimate on $\R$
\begin{equation}\label{eq:WB}
\left|x - \frac{\Phi_B}{\zeta_B^2}\right| |P_2'| \lesssim \frac {\rho^2} B.
\end{equation}
Indeed, by the definition of $\Phi_B$, we have for $x\geq 0$,
\[
\frac{\Phi_B}{\zeta_B^2} - x = \int_0^x \left( \frac{\zeta_B^2(y)}{\zeta_B^2(x)} - 1 \right) \udd y.
\]
Since $0\leq e^s-1\leq s e^s$ for any $s\geq 0$, we obtain for $x\geq 0$,
\begin{align*}
0 \leq \frac{\Phi_B}{\zeta_B^2} - x
&\leq \frac 2B \int_0^x \frac{\zeta_B^2(y)}{\zeta_B^2(x)} \left[|x|(1-\chi(x)) -|y|(1-\chi(y))\right] \udd y \\
&\leq \frac {2 x} B \int_0^x \frac{\zeta_B^2(y)}{\zeta_B^2(x)} \ud y
\leq \frac {1} B \frac{x^2}{\zeta_B^2(x)} .
\end{align*}
We obtain~\eqref{eq:WB} using $P_2'\in \cY$.

The spectrum of the operator $-\partial_{xx} + C \sech^2(\kappa x) = -\partial_{xx} + C \rho$
is known to contain exactly one eigenvalue for $0<C\leq 2 \kappa^2$ (see \emph{e.g.} \cite[Claim 4.1]{KMM}). Thus, for any odd function $\tilde v$, 
\[
\int (\partial_{x}\tilde v)^2 \geq 2 \kappa^2 \int \rho \tilde v^2.
\]
This implies~\eqref{old:LE5} for $B$ large enough.
\smallskip

Now, we claim the following estimate
\begin{equation}\label{old:LE8}
\|\rho \px v_1\|^2+\|\rho v_1\|^2\lesssim 
\|\rho^{\frac 12} \px\tilde v_1\|^2+\|\rho^{\frac 12} \tilde v_1\|^2+
\frac 1A \left(\|\sigma_A \px u_1\|^2+\frac{1}{A^2}\|\sigma_A u_1\|^2\right).
\end{equation}
Indeed, for $|x|<A$, we have $\chi_A(x)=1$ and so $\tilde v_1 = \chi_A\zeta_B v_1=\zeta_B v_1$. Thus,
for $B$ large,
\[
\int_{|x|<A} \rho^2 v_1^2 \lesssim \int_{|x|<A} \rho \zeta_B^2 v_1^2 \leq \int \rho \tilde v_1^2.\]
Moreover, for $|x|<A$, using $\px \tilde v_1 = \zeta_B' v_1 + \zeta_B \px v_1$ and
$|\zeta_B'|\lesssim B^{-1} \zeta_B$, we also have
\[
\int_{|x|<A} \rho^2 (\px v_1)^2 \lesssim \int_{|x|<A}\rho \zeta_B^2 (\px v_1)^2 
\lesssim \int\rho \left[ (\px \tilde v_1)^2 + \tilde v_1^2\right].
\]
Last, for $A$ large and $|x|\geq A$, one has $\rho^2 \lesssim \rho \sigma_A^2$ and using Lemma~\ref{LE:1},
\begin{align*}
\int_{|x|>A} \rho^2 \left[(\px v_1)^2 +v_1^2\right]
&\lesssim \rho(A) \left( \|\sigma_A\px v_1\|^2 + \|\sigma_A v_1\|^2 \right)\\
&\lesssim \rho(A)\varepsilon^{-2} \left(\|\sigma_A \px u_1\|^2+ \|\sigma_A u_1\|^2\right),
\end{align*}
and \eqref{old:LE8} is proved, for $A$ large, depending on $\varepsilon$.

To conclude the proof, we observe that \eqref{on:VB} is a consequence of \eqref{old:LE5} and \eqref{old:LE8}.
\end{proof}

From now on, we fix $B=B_0$. In the next three lemmas, we estimate $\widetilde K_1$, $K_2$ and $K_3$.

\begin{lemma}\label{LE:4}
It holds 
\[
|\widetilde K_1|\lesssim \frac 1{\sqrt{A}} \left(\|\sigma_A \px u_1\|^2+\frac{1}{A^2}\|\sigma_A u_1\|^2
+ \|\rho^2 u_1\|^2\right).
\]
\end{lemma}

\begin{proof}
We have
\[
|\chi_A'|\lesssim \frac {1}{A},\quad |\chi_A''|\lesssim \frac 1{A^2},
\quad |\chi_A'''|\lesssim \frac {1}{A^3}
\]
and
\[
\chi_A'(x)=\chi_A''(x)=\chi_A'''(x)=0 \quad \hbox{if~$|x|<A$ or if~$|x|>2A$.}
\]
Moreover, 
\[
|\zeta_B(x)|\lesssim Ce^{-\frac AB},\quad |\zeta_B'(x)|\lesssim \frac {1}B e^{-\frac AB} \quad \hbox{for~$|x|>A$.}
\]
Thus,
\[
|(\chi_A^2)'(\zeta_B^2)'|\lesssim \frac 1{AB} e^{-\frac AB} \sigma_A^2,\quad
(\chi_A')^2\zeta_B^2+|\chi_A''\chi_A| \zeta_B^2\lesssim \frac 1{A^2}e^{-\frac AB} \sigma_A^2.
\]
Using also~$|\Phi_B|\lesssim B$, we obtain
\[
|(\chi_A^2)'\Phi_B|\leq \frac{CB}{A}\sigma_A^2,\quad
|(\chi_A^2)'''\Phi_B|\leq \frac{CB}{A^3}\sigma_A^2.
\]
Thus (recall that~$B$ has been fixed)
\[
|\widetilde K_1|\lesssim \frac 1A \left( \|\sigma_A \px v_1\|^2
+ \frac 1{A^2}\|\sigma_A v_1\|^2\right).
\]
The estimate
\[
|\widetilde K_1|\lesssim \frac 1A \left(\varepsilon^{-2} \|\sigma_A \px u_1\|^2+\frac{1}{A^2}\|\sigma_A u_1\|^2
+ \|\rho^2 u_1\|^2\right)
\]
then follows from Lemma~\ref{LE:1}, applied to $u_1$.
Taking $A$ large depending on $\varepsilon$, the lemma is proved.
\end{proof}

\begin{lemma}\label{LE:5}
It holds 
\[
|K_2|\lesssim 
\varepsilon^{\frac 12} \left(\|\rho \px v_1\|^2+\|\rho v_1\|^2 \right).
\]
\end{lemma}
\begin{proof}
We recall the following technical estimate from \cite[Lemma~4.7]{KMMV}:
for $\varepsilon> 0$ small and $h\in L^2(\R)$,
\begin{equation}\label{tech2}
\|\rho^{-1} \Xe(\rho h)\| \lesssim \|h\|,
\quad \|\rho^{-1} \Xe \px (\rho h) \| \lesssim \varepsilon^{-\frac 12}\|h\|.
\end{equation}
Since 
\[ |\Psi_{A,B}|\lesssim B, \quad |\Psi_{A,B}'|\lesssim 1,\]
by the Cauchy-Schwarz inequality, one has
\[
|K_2|\lesssim B \left(\|\rho \px v_1\|+\|\rho v_1\|\right) \|\rho^{-1}\left[\Xe,P_2\right] U_1U_0 u_1\|.
\]
Now, we claim that for any $h\in L^2$,
\begin{equation}
\label{est:commut}
\| \rho^{-1} \left[\Xe,P_2\right] U_1U_0 h\| \lesssim \varepsilon^\frac12\|\rho S_\varepsilon h\|.
\end{equation}
Observe that this estimate applied to $h=u_1$ is sufficient to prove Lemma~\ref{LE:5}.

We prove \eqref{est:commut}. Setting
\[
f=\Xe P_2 U_1U_0 h, \qquad k =S_\varepsilon h
\]
we have
\[
-\varepsilon \partial_{xx} f+f= P_2 U_1U_1 h
\]
and
\[
-\varepsilon\partial_{xx}k + k =U_1U_0 h.
\]
From the latter, we obtain
\[
-\varepsilon \partial_{xx}(P_2 k)+P_2 k+2\varepsilon \px (P_2' k)-\varepsilon P_2'' k
=P_2 U_1 U_0 h.
\]
Combining the above identities we find
\[
-\varepsilon \partial_{xx}(f-P_2k)+(f-P_2 k)=2\varepsilon \px(P_2'k)-\varepsilon P_2'' k,
\]
and so
\begin{equation*}
\left[\Xe,P_2\right] U_1U_0 h =(f-P_2 k) 
 =\varepsilon \left[2 \Xe \partial_x (P_2' k)-\Xe (P_2'' k)\right].
\end{equation*}
Thus, estimate \eqref{est:commut} follows from $P_2',P_2''\in \cY$ and \eqref{tech2}.
\end{proof}

\begin{lemma}\label{LE:6}
For a constant $C>0$, it holds
\[
|K_3|\leq \frac\mu2 \left(\|\rho\px v_1\|^2+\|\rho v_1\|^2 \right)+ C |\zz|^4 +
\frac CA \left(\|\sigma_A \px u_1\|^2+\frac{1}{A^2}\|\sigma_A u_1\|^2\right) .
\]
\end{lemma}
\begin{proof}
We use the decomposition of $N$ from \eqref{eq:decN}, so that
\begin{align*}
K_3 & = z_1^2 \int \left(\Psi_{A,B} \partial_x v_{1}+\frac{1}{2}\Psi_{A,B}' v_1\right)S_\varepsilon R_0^\perp
 +\int \left(\Psi_{A,B} \partial_x v_{1}+\frac{1}{2}\Psi_{A,B}' v_1\right)S_\varepsilon (R_1^\perp+R_2^\perp)\\
&=K_{3,1}+K_{3,2}.
\end{align*}
On the one hand, since $R_0\in\cY$, by \eqref{tech2}, we have
\[
\|\rho^{-1} S_\varepsilon R_0^\perp\|\lesssim \|\rho^{-1} U_1U_0 R_0^\perp\|\lesssim |\zz|^2.
\]
Thus, by the Cauchy-Schwarz inequality,
\[
|K_{3,1}|\lesssim z_1^2 \left(\|\rho\px v_1\|+\|\rho v_1\|\right)\|\rho^{-1} S_\varepsilon R_0^\perp\| 
\lesssim \left(\|\rho\px v_1\|+\|\rho v_1\|\right) |\zz|^2.
\]
On the other hand, we observe that
\[|R_1|+|R_2| \lesssim |u_1|^2 +|u_1| |z_1|\rho^5 + |z_1|^3 \rho^{10}
\lesssim \delta \left(|u_1|+ |z_1|^2 \rho^{10}\right),\]
and so by Lemma~\ref{LE:1},
\[
\|\sigma_A S_\varepsilon (R_1^\perp+R_2^\perp)\|
\lesssim \varepsilon^{-1} \|\sigma_A (R_1^\perp+R_2^\perp)\|
\lesssim \delta \varepsilon^{-1} \left(\|\sigma_A u_1\|+|\zz|^2\right).
\]
Thus, by the Cauchy-Schwarz inequality,
\begin{align*}
|K_{3,2}|
&\lesssim \left(\|\sigma_A\px v_1\|+\|\sigma_A v_1\|\right)\|\sigma_A S_\varepsilon (R_1^\perp+R_2^\perp)\| \\
&\lesssim \delta \varepsilon^{-1} \left(\|\rho\px v_1\|+\|\rho v_1\|\right) \left(\|\sigma_A u_1\|+|\zz|^2\right) .
\end{align*}
This completes the proof of the lemma taking $\delta$ small enough, depending on $\varepsilon$ and $A$.
\end{proof}

From the identity \eqref{virial2} and Lemmas~\ref{LE:3}, \ref{LE:4}, \ref{LE:5} and~\ref{LE:6}, it follows that for $\varepsilon$ small, $A$ large and $\delta$ small,
\[
\dot \cK \leq -\frac \mu4 \left(\|\rho \partial_x v_1\|^2 + \|\rho v_1\|^2\right)
+ C|\zz|^4 + \frac C{\sqrt{A}} \left(\|\sigma_A \px u_1\|^2+\frac{1}{A^2}\|\sigma_A u_1\|^2
+\|\rho^2 u_1\|^2\right).
\]
Thus, by integration on $(0,T)$, and using (by Lemma~\ref{LE:1})
\[|\cK|\lesssim \|\sigma_A \px v_1\|^2 + \|\sigma_A v_1\|^2 + \|\sigma_A v_2\|^2
\lesssim \varepsilon^{-2} \delta^2\lesssim A \delta^2,
\]
we obtain
\[
\int_0^T \|\rho v_1\|^2 \ud t\lesssim A \delta^2 
+ \int_0^T|\zz|^4\ud t + \frac 1{\sqrt{A}} \int_0^T\left(\|\sigma_A \px u_1\|^2+\frac{1}{A^2}\|\sigma_A u_1\|^2
+\|\rho^2 u_1\|^2\right) \udd t.
\]
The estimate of Proposition~\ref{PR:3} now follows from Propositions~\ref{PR:1} and \ref{PR:2}.
\end{proof}

We fix $\varepsilon$ as in Proposition~\ref{PR:3}.

\section{Conclusion of the proof of Theorem~\ref{TH:1}}

Combining Lemma~\ref{LE:2} and Proposition~\ref{PR:3}, we find for any $T>0$,
\begin{equation*}
\int_0^T\|\rho^2 u_1 \|^2 \ud t \lesssim \int_0^T \|\rho S_\varepsilon u_1\|^2 \ud t
\lesssim A \delta^2 + \frac 1{\sqrt{A}} \int_0^T \|\rho^2 u_1 \|^2 \ud t.
\end{equation*}
Thus, for $A$ large enough, for any $T>0$,
\[
\int_0^T\|\rho^2 u_1 \|^2 \ud t \lesssim A \delta^2 .
\]
Now, we fix such an  $A$.
Let
\[
\cM = |\zz|^4+ \|\sigma_A \partial_x u_1\|^2 + \|\sigma_A u_1\|^2 
+ \|\sigma_A u_2\|^2.
\]
By Propositions \ref{PR:1} and \ref{PR:2}, it follows that
\[
\int_0^{+\infty} \cM(t) \udd t
\lesssim \delta^2.
\]
In particular, there exists $t_n\to +\infty$ such that 
$\lim_{n\to+\infty} \cM(t_n)=0$.
We conclude the proof of Theorem~\ref{TH:1} using a standard argument.
Using \eqref{eq:zu} and \eqref{eq:z2}, and then integration by parts, we compute
\begin{align*}
\dot \cM & = 2 |\zz|^2 \frac \ud{\ud t} |\zz|^2 + 2 \int \sigma_A \left[ (\px u_1) (\px \dot u_1)
+u_1 \dot u_1 + u_2 \dot u_2 \right]\\
& = - 4 |\zz|^2 \frac{z_2}\lambda \langle N, Y\rangle + 2 \int \sigma_A \left[ (\px u_1) (\px u_2)
+u_1u_2 - u_2 (L_0 u_1) -u_2 N^\perp\right] \\
& =- 4 |\zz|^2 \frac{z_2}\lambda \langle N, Y\rangle + 2 \int \left[ - \sigma_A'(\px u_1)u_2 + \sigma_Au_1u_2 (1-W''(H)) - \sigma_A u_2 N^\perp \right]
\end{align*}
Using $|\sigma_A'|\lesssim \sigma_A$ and $|N|\lesssim u_1^2 + z_1^2 Y^2$, we obtain
\[
\big| \dot \cM \big| \lesssim \cM.
\]
Let $t\geq 0$,
integrating on $(t,t_n)$, for $n$ large, we obtain
\[
\cM(t) \lesssim \cM(t_n) + \int_t^{t_n} \cM.
\]
Taking the limit $n\to \infty$,
\[
\cM(t)\lesssim \int_t^{+\infty} \cM.
\]
It follows that $\lim_{+\infty} \cM=0$, and thus for any bounded interval $I$ of $\R$,
\[
\lim_{t\to+\infty} \left( |\zz(t)|+\|\uu(t)\|_{H^1(I)\times L^2(I)}\right) = 0.
\]
The same result for $t\to-\infty$ is obtained using the time reversibility of the equation.

\section{Proof of Corollary~\ref{TH:2}}
We denote by $\omega_\eta$ and $H_\eta$ the analogues of $\omega$ and $H$ for the potential $W_\eta$.
For $\eta_0$ small enough, it is clear that the condition \eqref{on:W} is satisfied by $W_\eta$.
We observe that the setting is identically to the one in \cite[Proof of Corollary~1]{KMMV}.
In particular, we have for any $k\in \{0,\ldots,4\}$, for any $x\in \R$,
\[
|H_\eta^{(k)}(x) - H^{(k)}(x)|\lesssim \eta_0 e^{-\frac 14 \omega |x|},\quad
|\omega_\eta-\omega|+\sup_\R|W_\eta^{(k)}(H_\eta)-W^{(k)}(H)|\lesssim \eta_0.
\]
(In this proof, all implicit constants are independent of $\eta_0$.)
It follows from Hypothesis~1 for $W$ and standard perturbation arguments that
the operator $L_{0,\eta}=-\partial_{xx}+W''_\eta(H_\eta)$ also has a second eigenvalue $\lambda_\eta\in (\frac 12\omega_\eta,\omega_\eta)$ (see Remark~\ref{RK:1}) associated to an odd eigenfunction~$Y_\eta$, which means that Hypothesis 1 is satisfied for $W_\eta$. Moreover, it holds $|\lambda_\eta-\lambda|\lesssim \eta_0$.
We denote by $P_{1,\eta}$ and $P_{2,\eta}$ the analogues of $P_1$ and $P_2$ for the potential $W_\eta$. By standard ODE arguments, one also obtains the following estimates, for any $k=0,1,2,3$, $x\in \R$,
\[
|Y_\eta^{(k)}(x)|+
|P_{1,\eta}'(x)|+|P_{2,\eta}'(x)|\lesssim e^{-\kappa |x|}.
\]
Moreover, for any $k=0,1,2,3$, $k'=0,1$,
\[
\sup_\R|Y_\eta^{(k)}-Y^{(k)}|+
\sup_\R|P_{1,\eta}^{(k')}-P_1^{(k')}|+\sup_\R|P_{2,\eta}^{(k')}-P_2^{(k')}|\lesssim \eta_0.
\]
Concerning Hypothesis~2, we define $g$ and $g_\eta$ as the respective solutions of 
$L_0 g = 4 \lambda^2 g$ and
$L_{0,\eta} g_\eta = 4 \lambda_\eta^2 g_\eta$
with $g(0)=g_\eta(0)=0$ and $g'(0)=g_\eta'(0)=1$. Since $4 \lambda^2 > \omega^2$ and $4 \lambda_\eta^2 > \omega_\eta^2$, the functions $g$ and $g_\eta$ are bounded. Moreover, by a standard perturbation argument, on any compact interval $I$ of $\R$, one has
$|g_\eta-g|\lesssim C_I\eta_0$. Thus, Hypothesis~2 for $W_\eta$ holds
for $\eta_0$ small enough. Last, by Hypothesis~3 for the potential $W$, there exists a constant $\gamma>0$ such that the 
operator $-(1-\gamma) \partial_{xx} + \frac 12 x P_2'$ has at most one negative eigenvalue.
This means that for any odd function $u\in L^2$,
\[
(1-\gamma)\int (u')^2 \geq \frac 12 \int xP_2' u^2.
\]
Now, by the estimates on $P_{2,\eta}'$ and a standard argument (see the proof of Lemma~\ref{LE:3}), for
$\eta_0$ small enough, we have, for any odd function $u\in L^2$,
\[
\gamma \int (u')^2 \geq \int |xP_2'-x P_{2,\eta}'| u^2.
\]
Thus, for any odd function $u\in L^2$,
\[
\left(1-\frac \gamma2\right)\int (u')^2 \geq \frac 12 \int xP_{2,\eta}' u^2,
\]
which means that Hypothesis 3 holds for $W_\eta$.

\end{document}